\documentclass[12pt]{amsart}
\usepackage{amscd}
\usepackage{amssymb}
\usepackage{a4wide}
\usepackage{amstext}
\usepackage{amsthm}
\usepackage{mathrsfs}
\usepackage[final]{hyperref}
\hypersetup{unicode= false, colorlinks=true, linkcolor=blue,
anchorcolor=blue, citecolor=green, filecolor=red, menucolor=blue,
pagecolor=blue, urlcolor=blue} 
\numberwithin{equation}{section}

\newcommand{\N}{\mathbb{N}}

\newcommand{\CC}{\mathbb {C}}

\newcommand{\PP}{\mathcal{P}}
\newcommand{\ZL}{\mathcal{Z}}
\newcommand{\RR}{\mathbb{R}}
\newcommand{\Z}{\mathbb{Z}}

 \DeclareMathOperator{\dist}{dist}

\DeclareMathOperator{\supp}{supp}

\DeclareMathOperator{\conv}{conv}

\DeclareMathOperator{\card}{card}

\renewcommand{\phi}{\varphi}

\newcommand{\pw}{\mathcal{P}W_\pi}

\newcommand{\vep}{\varepsilon}

\newtheorem{Thm}{Theorem}[section]
\newtheorem{theorem}[Thm]{Theorem}
\newtheorem{lemma}[Thm]{Lemma}
\newtheorem{proposition}[Thm]{Proposition}
\newtheorem{corollary}[Thm]{Corollary}
\newtheorem{remark}[Thm]{Remark}
\renewcommand{\qedsymbol}{$\square$}
\begin{document}
\sloppy
\title[ Subspaces of $C^\infty$ invariant under the differentiation]
{Subspaces of $C^\infty$ invariant under the differentiation}
\author{Alexandru Aleman, Anton Baranov, Yurii Belov}

\address{
Alexandru Aleman,
\newline
Centre for Mathematical Sciences, Lund University,
\newline
P.O. Box 118, SE-22100
Lund, Sweden
\newline{\tt aleman@maths.lth.se}
\newline\newline \phantom{x}\,\,
 Anton Baranov,
\newline
Department of Mathematics and Mechanics,
St. Petersburg State University,
\newline
St. Petersburg, Russia,
\newline
\phantom{x}\,\, and
\newline
National Research University  Higher School of Economics,
\newline
St. Petersburg, Russia,
\newline {\tt anton.d.baranov@gmail.com}
\newline\newline \phantom{x}\,\, Yurii Belov,
\newline
Chebyshev Laboratory,
St. Petersburg State University,
\newline
St. Petersburg, Russia,
\newline {\tt j\_b\_juri\_belov@mail.ru}
\newline\newline \phantom{x}
}
\thanks{The second and the third author were supported by the Chebyshev Laboratory
(St.Petersburg State University) under RF Government grant 11.G34.31.0026 and by JSC "Gazprom Neft". The third author was partially supported by RFBR grant 12-01-31492 and Dmitriy Zimin fund Dynasty}

\begin{abstract}

Let $L$ be a proper differentiation invariant subspace of
$C^\infty(a,b)$ such that the restriction operator
$\frac{d}{dx}\bigl{|}_L$ has a discrete spectrum $\Lambda$
(counting with multiplicities). We prove that $L$ is spanned by
functions vanishing outside some closed interval $I\subset(a,b)$
and monomial exponentials $x^ke^{\lambda x}$ corresponding to
$\Lambda$ if its density does not exceed  the critical value
$\frac{|I|}{2\pi}$, and moreover, we show that the result is not
necessarily true when  the density of $\Lambda$ equals the
critical value. This answers a question posed by the first author
and B. Korenblum. Finally, if the residual part of $L$ is trivial, then $L$
is spanned by the monomial exponentials it contains.

\end{abstract}

\maketitle
\section{Introduction}  Consider the  space  $C^\infty(a,b)$
equipped with the usual topology of uniform convergence on compacta of
each derivative $f^{(k)}, k=0,1,\ldots$; more specifically, the
topology  given by any of the translation invariant metrics given
below. Consider a sequence $(I_j)$ of compact intervals with $\cup_j
I_j=(a,b)$, denote by $\|\cdot\|_j$ the sup-norm over $I_j$ and set
$$d(f,g)=\sum_{j,k=0}^\infty 2^{-j-k}
\frac{\|f^{(k)}-g^{(k)}\|_j}{1+\|f^{(k)}-g^{(k)}\|_j}.
$$

The present paper concerns the structure of closed subspaces of $C^\infty(a,b)$ which are invariant for the differentiation operator $D=\frac{d}{dx}$. Our  investigation follows the classical line, namely we are going to consider an appropriate version of  {\it  spectral synthesis} for these subspace, which we now explain.

Strictly speaking, a continuous operator has the property of spectral syntesis if any non-trivial invariant subspace is generated by the   root-vectors contained in it. The definition extends in an obvious way to families of commuting operators. The parade examples are the translation invariant subspaces of the locally
convex space of continuous functions on  the real line. These  are  now well
understood due to the work of J. Delsarte \cite{Del}, J.-P. Kahane
\cite{Kahane} and L. Schwartz \cite{Schw}. In  the setting of entire
functions  translation-invariance is equivalent to complex differentiation invariance and the spectral synthesis property has been proved  by  L. Schwartz \cite{Schw1}.

The structure of differentiation-invariant subspaces of  $C^\infty(a,b)$  is more complicated and was only investigated recently in \cite{alkor}. The reason for the additional complication is the presence of  the following subspaces: Given a closed set $S\subset (a,b)$ let
\begin{equation}\label{complication}L_S=\{f\in  C^\infty(a,b):~f^{(k)}(S)=\{0\},\, k\ge 0\}\,.\end{equation}
In many cases these subspaces are nontrivial,  and obviously, they contain no root-vector of $D$, since these functions  are monomial exponentials, i.e. they have the form $x\to x^ne^{\lambda x}\,,n\in \mathbb{N}\,, \lambda \in \mathbb{C}\,.$ According to \cite{alkor} the $D$-invariant subspaces of  $C^\infty(a,b)$ can be classified in terms of  the spectrum of the restriction of this operator. More precisely, given such a closed subspace $L$ of   $C^\infty(a,b)$  we have the following three alternatives:

(i) $\sigma(D|_L)=\mathbb{C}$,

(ii)  $\sigma(D|_L)=\emptyset$,

(iii)  $\sigma(D|_L)$ is a nonvoid discrete subset of $\mathbb{C}$ consisting  of eigenvalues of $D$.

Very little is known about the structure of subspaces of the form
(i). A concrete example is obtained by choosing the set $S$ in
\eqref{complication} to consist of finitely many points, or
disjoint intervals. The subspaces of type (ii) are called residual
and  are completely characterized in  \cite{alkor}. The main
result of that paper asserts that such a subspace has the form
$$L_I=\{f\in C^\infty(a,b): f^{(k)}(I)=0, k\geq0\}$$ for some
interval $I\subset(a,b)$ which is { \it relatively closed in $(a,b)$.} The interval $I$ may reduce to one point. We will call $I$ {\it the
residual interval}.

The subspaces of type (iii) are the main concern of this paper. Such a subspace $L$ might have a nontrivial residual part (see again  \cite{alkor} for the details) given by
$$L_{res}=\bigcap\{p(D)L:~p-\text{polynomial}\,\}.$$
The natural question which arises and has been formulated in  \cite{alkor} is:

{\it Is every  $D$-invariant subspace of type {\rm (iii)}
generated by its residual part and the monomial exponential it contains?}

In the case when  the spectrum of the restriction of $D$ is a
finite set an affirmative answer has been given in
\cite[Proposition 6.1]{alkor}. The general case
is quite subtle and, surprisingly enough, our results reveal an interplay between the length of the residual interval  and the uniform upper density of the set  $\sigma(D|_L)$. This is defined as usual by
$$
\mathcal{D}_+(\sigma(D|_L)):=\lim_{r\rightarrow\infty}\sup_{x\in\mathbb{R}}\frac{\card\{\lambda\in \sigma(D|_L): \Re\lambda\in[x,x+r]\}}{r}\,,
$$
where multiplicities are counted.

Our main results are given below. Throughout in the statements  $L$ is a $D$-invariant subspace of $C^\infty(a,b)$ satisfying alternative (iii) from above and $\mathcal{E}(L)$ denotes the set of monomial exponential functions $x\to x^ke^{\lambda x}$ contained in $L$.

\begin{theorem}
\label{main1}
 Assume that the $D$-invariant subspace $L$ has a compact  residual interval $I$. If
$$
2\pi \mathcal{D}_+(\sigma(D|_L))<|I|\,,
$$
then
$$
L= \overline{L_{res}+\mathcal{E}(L)}.
$$
\end{theorem}

The restriction on the density of $\sigma(D|_L)$ turns out to be essential.
If $2\pi \mathcal{D}_+(\sigma(D|_L)) = |I|$, then the spectral synthesis may fail as the next theorem shows.

\begin{theorem}
\label{main2}
There exists a $D$-invariant subspace $L$ as above
such that
$$
L_{res} = \{f\in C^\infty(-2\pi,2\pi):
f|_{[-\pi, \pi]} \equiv 0\}\,,$$ $$\mathcal{D}_+(\sigma(D|_L)) = 1\,,
$$
but
$$
L \ne \overline{L_{res}+\mathcal{E}(L)}.
$$
\end{theorem}
{
Our methods pertain also to the case when the residual interval is not compact. Note that in this case we have  either  $I=(a,b)$, in which case the  residual subspace is trivial $L_{res}=\{0\}$, or there exists $c\in (a,b)$ such that $I=(a,c]$, or $I=[c,b)$.
 It turns out that in these cases spectral synthesis does always hold, as the following theorem shows.}
 \begin{theorem}
If the residual interval $I$ of the $D$-invariant subspace $L$ is non-compact, then
$$L= \overline{L_{res}+\mathcal{E}(L)}.$$
\label{main3}
\end{theorem}
 
Our approach is a substantial improvement of the method used in \cite{alkor} and is inspired by the ideas in  \cite{BBB1}.
It is well known that the spectral synthesis problem
for linear operators in a Hilbert space
is closely related to the {\it hereditary completeness} property
for systems of vectors (see e.g. \cite{markus}). The complete and minimal system of
vectors $\{x_n\}_{n\in\mathbb{N}}$ with complete
biorthogonal $\{y_n\}_{n\in\mathbb{N}}$ is said to be {\it hereditarily complete}
if any mixed system $\{x_n\}_{n\in \mathbb{N}\setminus N_1}\cup\{y_n\}_{n\in N_1}$
is also complete.  This property for exponentials in the  space $L^2(I)$
was extensively studied in \cite{BBB2}.
As we will see later we need to prove completeness of {\it some}
mixed system (the whole system $\{x_n\}$ in our case is not even complete).
This was done in \cite{BBB2} using the results of \cite{BBB1} and sharp density result
of Beurling--Malliavin type. The Proposition \ref{mainprop} is an adapted
 version of the main result of \cite{BBB2}. So the proof of the Theorem \ref{main1}
consists of two steps: we reduce the problem to a Hilbert space problem (Section \ref{step1})
and solve the appropriate mixed completeness problem (Section \ref{step2}).

The idea of the construction of the counterexample in Theorem \ref{main2}
goes back to \cite[Theorem 1.3]{BBB1}.
\smallskip

\textbf{Organization of the paper.} The paper is organized as follows.
Theorem \ref{main1} is proved in Sections \ref{step1} and \ref{step2}.
Example of the absence of spectral synthesis is given in Section
\ref{examp}. In Section \ref{nores} we prove Theorem \ref{main3}.


\section{Proof of Theorem \ref{main1}
\label{step1}}

\subsection{Preliminaries}
The continuous linear functionals on $C^\infty(a,b)$  are compactly supported
distributions of finite order on the interval $(a, b)$ (that is, distributions of the form
$\phi=f^{(n)}$,  $n\in\mathbb{N}$, where $f\in L^1_{loc}(a,b)$ and differentiation
is understood in the sense of distributions). We denote this class by $S(a,b)$.
As usual we can define the
Fourier transform of $\varphi\in S(a,b)$ by the formula
$$
\hat{\varphi}(z)=\varphi(e^{iz}).
$$
Since $\phi$ has finite order, the function $\hat{\varphi}$ is an entire function
of finite exponential
type with at most polynomial growth on the real line. Put
$$
\mathcal{H}_I=\biggl{\{}f: f = \sum_{k=0}^n z^k(f_k(z)+c_k),
\quad f_k\in\mathcal{PW}_I, \ c_k \in \CC \biggr{\}},
$$
where $\mathcal{PW}_I$ is the Fourier image of the space $L^2(I)$.
For every functional $\varphi\in S(a,b)$ we have $\hat{\varphi}\in\mathcal{H}_I$,
where $I=I(\varphi)$ is a closed interval such that $\supp\phi\subset I$.

Under the  Fourier transform the duality  between $S(a,b)$ and $C^\infty(a,b)$ becomes
the duality between the space $\mathcal{H}=\cup_{I\subset(a,b)}\mathcal{H}_I$ and
the space of entire functions $\mathcal{U}$ with conjugate indicator diagrams
in $(a,b)$ and decreasing faster than any polynomial along the real line.
Let $F\in\cup_{I\subset(a,b)}\mathcal{H}_I$, $G\in\mathcal{U}$. If $F$ has
infinite number of zeros, then its bilinear
form $(F,G)_{\mathcal{H},\mathcal{U}}=(\check{F},\check{G})_{S(a,b), C^{\infty}(a,b)}$
can be viewed as usual inner product in $L^2(\mathbb{R})$ (or $\mathcal{PW}_I$)
of the functions $F(z)\slash P(z)$ and $G(z)\overline{P(\overline{z})}$,
where $P$ is a polynomial of a sufficiently large degree whose zero set is a subset of
zero set of $F$. This explains how spaces $\mathcal{PW}_I$ come into the picture.

For an entire function $F$ we will denote its zeros set as $\mathcal{Z}_F$.

\subsection{The associated  a Hilbert space problem}
Let us denote $\Lambda=\sigma(D|_L)$,
$$
L_0= \overline{L_{res}+\mathcal{E}(L)},
$$
and let $I$ be the {\it residual interval} given by
$$
L_{res}= L_I =\{f\in C^\infty(a,b): f^{(k)}(I)=0\,,k\ge 0\}.
$$
If $I$ reduces to one point, it is easy to see that
$L=C^\infty(a,b)$. Thus we shall assume throughout that
$I$ does not reduce to a point.

Let us consider the annihilators $L^\perp$ and $L_0^\perp$
of $L$ and $L_0$ respectively in the dual space $S(a,b)$.
It is easy to see that
$$
\widehat{L_0^{\perp}}=\{F: F\in \mathcal{H}_I, F\bigl{|}_\Lambda=0\}.
$$
It will be  sufficient  to prove that $L^{\perp}$ is dense in $L^\perp_0$,
in the weak star topology.

Let us consider the distribution $\phi\in L^\perp$ with
the maximal length of $\conv\supp\phi$. It is clear that $\hat{\phi}\in \mathcal{H}_I$
but $\hat{\phi}\notin \mathcal{H}_J$ for any subinterval $J\subset I$, $J\neq I$.  We can write
$$
\hat{\phi}(z)=G_\Lambda(z)E(z),
$$
where $G_\Lambda$ is some canonical product corresponding to the sequence
$\Lambda$) and $E$ is some entire function. In fact, we shall choose $G_\Lambda$ such that $G(z)/\overline{G(\bar{z})}$
is a quotient of two Blaschke products. Moreover, without loss of generality we can assume that $E$ has no multiple zeros and $E(\lambda)\neq0$ for $\lambda\in\Lambda$.

From \cite[Proposition 3.1]{alkor} we know that
$$
\frac{G_\Lambda(z)E(z)}{z-w}\in \widehat{L^\perp},\quad w\in\mathcal{Z}_E.
$$
Therefore, we can further assume that $G_\Lambda E\in\mathcal{PW}_I$, otherwise we  can start with the function
$\frac{E(z)}{(z-w_1)...(z-w_n)}$ in place of $E$.

We argue by contradiction. Suppose that  $L^{\perp}$ is not dense in $L_0^\perp$.
Then there exists an entire function $T$ such that
$G_\Lambda T\in \mathcal{H}_I$ and $G_\Lambda T\notin \widehat{L^\perp}$.
We fix such a function $T$ and number $N$ such that $G_\Lambda(x) T(x) =
O(1+|x|^N)$.
From the Hahn--Banach
Theorem we know that there exists a non-trivial function $f\in C^\infty(I)$ such that
$$
\biggl{(}\frac{G_\Lambda(z)E(z)}{z-w},\hat{f}\biggr{)}=0,
\quad w\in\mathcal{Z}_E \quad \text {and}\quad (G_\Lambda T, \hat{f})\neq 0.
$$

In order to arrive at a Hilbert space setting we assume first that  $G_\Lambda T\in \mathcal{PW}_I$.  In this case both equations may be understood as usual inner products in $L^2(\mathbb{R})$. The general case   ($G_\Lambda T$ grows polynomially)  will be reduced to
this special situation in the next subsection.
This leads to the following system of equations
\begin{equation}
\begin{cases}
\int_\mathbb{R}\frac{G_\Lambda(x)E(x)}{x-w}\overline{F(x)}dx=0,\quad w\in\mathcal{Z}_E,\\
\int_\mathbb{R}G_\Lambda(x)T(x)\overline{F(x)} dx \neq 0.
\label{inteq}
\end{cases}
\end{equation}

The contradiction we are seeking for is  given by the following proposition. The reproducing kernel at $\lambda$ in $\mathcal{PW}_I$ will be denoted by $k_\lambda$.
\begin{proposition}
\label{pw}
Let $G\in \mathcal{PW}_I$ be a function with simple zeros and such that $G\notin \mathcal{PW}_J$ for any proper subinterval $J$ of $I$. If we have a partition of its zero set $\mathcal{Z}_G=\Lambda_1\cup\Lambda_2$, $\Lambda_1\cap\Lambda_2=\emptyset$ such that $2\pi \mathcal{D}_+(\Lambda_2)<|I|$,
then the mixed system
\begin{equation}
\{k_\lambda\}_{\lambda\in\Lambda_2}\cup\biggl{\{}\frac{G(z)}
{z-\lambda}\biggr{\}}_{\lambda\in\Lambda_1}
\label{mixed}
\end{equation}
is complete in $\mathcal{PW}_I$.
\label{mainprop}
\end{proposition}

Let us assume for the moment that Proposition \ref{mainprop} is proved, and that  $G_\Lambda T\in \mathcal{PW}_I$.
From \eqref{inteq} it follows that $F$
is orthogonal (in $\mathcal{PW}_I$) to the family
$\bigl{\{}\frac{G_\Lambda(z)E(z)}{z-w}\bigr{\}}_{w\in\mathcal{Z}_E}$.
Apply Proposition \ref{pw} to the function $G=G_\Lambda E$,
with $\Lambda_2=\Lambda, \Lambda_1=\mathcal{Z}_E$,
to conclude that $F$ belongs to the closed span of $
\{k_\lambda\}_{\lambda\in\Lambda}$, which obviously  contradicts  the second equation in \eqref{inteq}.

\subsection{Reduction of the general case to the Hilbert space setting.}
In the general case  we only know that that $G_\Lambda T(x)=O(1+|x|^N)$ on the real line.  Put
$$
M=\{G_\Lambda H: G_\Lambda H\in \mathcal{PW}_I\}.
$$
By the previous argument we have that  $\widehat{L^{\perp}}$ is dense in $M$, hence,  it remains to prove that $M$ is dense in $\widehat{L^{\perp}_0}$. Assume the contrary.
Then there exists a function $f\in C^{\infty}(a,b)$ such that $\hat{f}\perp M$ but $(G_\Lambda T,\hat{f})\neq0$.

Let us fix a finite set $W$, $W\cap\Lambda=\emptyset$,
such that there exists a function $g$ of the form $\sum_{w\in W}c_we^{iw t}$ with the property that for  $0\leq k\leq 2N+2$, $g^{(k)}-f^{(k)}$ vanishes at the endpoints of $I$. Moreover,  it is clear that there exists $G_\Lambda F_0\in M$ such that $G_\Lambda(T+F_0)$ vanishes on $W$.   Thus $G_\Lambda(T+F_0)=G_\Lambda P_W T_1$, where $P_W$ is a polynomial vanishing on $W$. Obviously,
$(G_\Lambda P_W T_1, F)\neq 0$.

Now let  $\tilde{F}=\widehat{f-g}$,  and note that
$\tilde{F}(x)=O((1+|x|)^{-2N})$;
in particular, $\tilde{F}Q\in L^2(\mathbb{R})$, for every polynomial $Q$  of degree strictly less than $2N$.

For every entire function $U$ with  $G_\Lambda P_W U\in\mathcal{PW}_I$,
we have the system
\begin{equation}
\begin{cases}
\int_\mathbb{R}G_\Lambda(x)P_W(x)U(x)\overline{\tilde{F}(x)}dx=0,\\
\int_\mathbb{R}G_\Lambda(x)P_W(x)T_1(x)\overline{\tilde{F}(x)}dx\neq0.
\end{cases}
\end{equation}

Now  fix a function $U$ such that $T_1$ and $U$ have at least $N+2$ common zeros and $G_\Lambda U\notin \mathcal{PW}_J$ for any proper subinterval $J\subset I$  . Then write $U=QU_1$, $T_1=QT_2$, where $Q$ is a polynomial of degree $N+2$ and let $Q^*(z)=\overline{Q(\overline{z})}$.  We then have
$$
(G_\Lambda P_W T_2, \tilde{F}Q^*)\neq 0, \qquad \biggl{(}\frac{G_\Lambda P_W U_1}{z-u}, \tilde{F}Q^*\biggr{)}= 0,
\quad \text{if} \quad u\in\mathcal{Z}_{U_1},
$$
which is a system of the form \eqref{inteq}
with the set $\Lambda\cup W$  instead of $\Lambda$ and with $U_1$
instead of $E$. As we have seen in the previous subsection
this system leads to a contradiction. \qed


\section{Proof of Proposition \ref{pw}\label{step2}}
In order to complete the proof of Theorem \ref{main1} it remains to prove Proposition \ref{mainprop}. To simplify notations we assume without loss that $I=[-\pi,\pi]$ and write  $\mathcal{PW}_\pi$ instead of $\mathcal{PW}_{[-\pi,\pi]}$. We denote by  $k_\lambda$ the reproducing kernel of $\mathcal{PW}_\pi$
corresponding to the point $\lambda$, that is,
$$
k_\lambda(z) = \frac{\sin \pi (z-\overline \lambda)}{\pi(z-\overline
\lambda)}, \qquad \text{and}\quad f(\lambda) = (f,k_\lambda).
$$
Recall that for any $\gamma\in \RR$ the system
$\{k_{n+\gamma}\}_{n\in \mathbb{Z}}$
is an orthogonal basis of $\mathcal{PW}_\pi$.

\subsection{Equation for the function of zero exponential type}
The proof is similar to the proof of the main theorem in \cite{BBB1}.
Assume the contrary. Then there exists  a nonzero  $h\in\mathcal{PW}_\pi$
such that
\begin{equation}
\biggl{(}h,\frac{G(z)}{z-\lambda}\biggr{)}=0,\quad \lambda\in\Lambda_1,\qquad (h,k_\lambda)=0,\quad \lambda\in\Lambda_2.
\label{inn}
\end{equation}
We expand $h$ with respect to the orthogonal basis $\{k_n\}_{n\in\mathbb Z}$:
$$
h(z)=\sum_n\overline{a_n}k_n(z).
$$
Then the equations \eqref{inn} can be rewritten as
\begin{equation}
\begin{cases}
\sum_n\frac{a_nG(n)}{\lambda-n}=0,\quad \lambda\in\Lambda_1,\\
\sum_n\frac{(-1)^n\overline{a_n}}{\lambda-n}=0,\quad \lambda\in\Lambda_2.
\end{cases}
\end{equation}
Then there exist entire functions $S_1$ and $S_2$ such that
\begin{equation}
\begin{cases}
\sum_n\frac{a_nG(n)}{z-n}=\frac{G_1(z)S_1(z)}{\sin \pi z}\\
\sum_n\frac{(-1)^n\overline{a_n}}{z-n}=\frac{G_2(z)S_2(z)}{\sin \pi z}=\frac{h(z)}{\sin \pi z},
\label{main}
\end{cases}
\end{equation}
where $G_1$ and $G_2$ are canonical products  corresponding to $\Lambda_1,\Lambda_2$, respectively.  The functions $S_1$, $S_2$ satisfying \eqref{main} parametrize all functions orthogonal to the mixed system \eqref{mixed}.
Put $V=S_1S_2$. Comparing the residues in equations \eqref{main} at the points $n$ we get
$$V(n)=(-1)^n|a_n|^2.$$
Therefore we have the representation
$$V(z)=Q(z)+R(z) \sin \pi z,$$
where
$$
Q(z)=\sin \pi z \sum_n\frac{|a_n|^2}{z-n}
$$
and $R$ is a function of zero exponential type. Without loss of generality we can assume
that $S_1$ and $S_2$ are real on the real line (the similar formulae hold for $S_1+S^*_1$ and $S_2+S_2^*$, see \cite{BBB1}).
We know that
\begin{equation}
R(\lambda)+\sum_n\frac{|a_n|^2}{\lambda-n}=0,\quad \lambda\in\mathcal{Z}_{S_2}.
\label{R0}
\end{equation}

This is a very restrictive condition because we can start with
a basis $\{k_{n+\gamma}\}_{n\in\mathbb{Z}}$, $\gamma\in[0,1)$ which is sufficiently far from real the zeros of $S_2$  so that the Cauchy transform of the sequence $|a_n|^2$ is not very big on real zeros of $S_2$. On the other hand, there are a lot of real zeros of $S_2$ because the function $V$ has at least one zero in each interval $[n, n+1)$. In the next subsection we present this idea  in detail.

\subsection{Choice of the basis.} Choosing amongst the bases
$\{k_{n+\gamma}\},\gamma\in \mathbb{R}$, is of course equivalent
to  the corresponding translation to the functions involved. For
simplicity, we shall keep the same notations for these. Then we can find
a sufficiently small $\delta>0$ for which there exist two subsets
$\Sigma, \Sigma_1$ of the zero set $\mathcal{Z}(S_2)$ of the function
$S_2$ with the following properties:
\begin{itemize}
\begin{item}
$\Sigma$ has exactly one point in those intervals
where $\mathcal{Z}(S_2)\cap[n, n+1)\neq\emptyset$,
and
$$
\dist(x,\mathbb{Z})>\frac{\delta}{1+x^2}, \qquad x\in\Sigma;
$$
\end{item}
\begin{item}
$\Sigma_1$ has positive upper density,
and $\dist(x, \mathbb{Z})>\delta$, $x\in\Sigma_1$.
\end{item}
\end{itemize}

We need to consider three cases.
If $R$ is a nonzero polynomial, then the zeros of the
function \eqref{R0} approach $\mathbb{Z}$ and we obtain
a contradiction to the existence of $\Sigma_1$.
If $R=0$, then it is known that the density of $\Sigma_1$ is zero
\cite[Proposition 3.1]{BBB1}. Finally, if $R$ is not a polynomial,
we can divide it by $(z-z_1)(z-z_2)$, where $z_1$ and $z_2$
are two arbitrary zeros of $R$, $z_1,z_2\not\in\Sigma$,
to get a function $R_1$ of zero exponential type which is
bounded on $\Sigma$.

Next, we obtain some information on $\Sigma$.
For a discrete set $X=\{x_n\} \subset \mathbb{R}$ we
consider its counting function
$n_X(t) = {\rm card}\, \{n: x_n \in [0, t)\}$, $t\ge 0$,
and $n_X(t) = -{\rm card}\, \{n: x_n \in (-t, 0)\}$, $t<0$.
If $f$ is an entire function and $X$ is the set of its real zeros
(counted according to multiplicities), then there exists a branch
of the argument of $f$ on the real axis,
which is of the form $\arg f(t) = \pi n_X(t) +\psi(t)$, where $\psi$
is a smooth function. Such choice of the argument is unique
up to an additive constant and in what follows we always assume that the argument is chosen
to be of this form.

Denote by $\tilde u$ the conjugate function (the Hilbert
transform) of $u$,
$$
\tilde u (x)=\frac{1}{\pi} v.p.
\int_\mathbb{R} \bigg(\frac{1}{x-t} +\frac{t}{t^2+1}\bigg) u(t)dt.
$$

We use the fact that for every function $f\in\pw$ with the
conjugate indicator diagram $[-\pi,\pi]$ and all zeros
in $\overline{\mathbb{C}_+}$, one has
\begin{equation}
\label{argum}
\arg f = \pi x +\tilde{u}+c,
\end{equation}
where $u\in L^1((1+x^2)^{-1}dx)$, $c\in\mathbb{R}$.

Indeed, the function $g=e^{-i\pi z}\overline{f(\overline{z})}$ is an {\it outer} function in $\mathbb{C}^+$ and, hence, it can be represented
in the form $e^{\log|g|+i\widetilde{\log|g|}}$. Taking into account the arguments we
get  \eqref{argum}.

It follows from \eqref{main} that $GV\in\mathcal{P}W_{2\pi}$. For any $f\in\mathcal{PW}_\pi$ we put
$$f^{\#}(z)=f(z)B^-(z),$$
where $B^-(z)$ is a Blaschke product in $\mathbb{C}^-$ with zero set $\mathcal{Z}_f\cap\mathbb{C}^-$. The function $f^{\#}$ has no zeros in $\mathbb{C}^-$ and is also in $\mathcal{PW}_\pi$. So, we have
$$h^{\#}=G^{\#}_2S^{\#}_2\in\pw,\quad G^{\#}V^{\#}\in\mathcal{P}W_{2\pi}.$$
Using these inclusions and the fact that $V^{\#}$ has at least one zero in each
interval $(n,n+1)$ we find an equation
on the counting function of $\Sigma$.
In the next subsection we will show that this contradicts to the fact that
nonconstant entire function of zero exponential type is bounded on $\Sigma$.

Let us consider its representation $V^\# = V_0 H$, where the zeros
of $V_0$ are simple, interlacing with $\mathbb{Z}$ and
$V_0|_\Sigma=0$. It is clear that $\arg V_0 = \pi x+ O(1)$. Since,
$$
\arg(G^\# V^\#)=  2\pi  x + \tilde{u_1}+c,
$$
and
$$\arg(G^\#)=  \pi  x + \tilde{u_2}+c,$$
we conclude that
$$
\arg(H)= \tilde{u_3}+O(1).
$$

Consider the equality $h^\# = G_2^\# H S_2^\#/H$ and note that
$$
\arg\Big(\frac{S_2}{H}\Big)  = \pi n_\Sigma -\alpha,
$$
where $\alpha$ is some nondecreasing function on $\mathbb{R}$.
This follows from the fact that
$S_2^\#/H$ vanishes only on a subset of the real axis which contains $\Sigma$.
Applying the representation (\ref{argum}) to $h^\#$, we conclude that
\begin{equation}
\label{n15}
\arg G_2^\# + \pi n_\Sigma(x) = \pi x + \tilde{u}+ v + \alpha,
\end{equation}
where $u\in L^1((1+x^2)^{-1}dx)$, $v\in L^\infty(\mathbb{R})$,
and $\alpha$ is nondecreasing.

Using the fact that the upper density of $\Lambda_2$ is less than $\pi$ we get an equation

\begin{equation}
\label{n16}
\pi n_\Sigma(x) = \pi \varepsilon x + \tilde{u}+ v + \alpha_1,\quad \varepsilon >0,
\end{equation}
and $\alpha_1$ is nondecreasing.

Summing up, we have an entire function $R$ of zero exponential type which is not a polynomial, and which is bounded
on a set $\Sigma\subset\mathbb R$ satisfying \eqref{n16}.

\subsection{Beurling--Malliavin meets P\'olya}
To deduce a contradiction from \eqref{n16}, we use some information on the classical P\'olya problem
and on the second Beurling--Malliavin theorem.
We say that a sequence $X=\{x_n\} \subset \mathbb{R}$ is a {\it P\'olya sequence}
if any entire function of zero exponential type which is bounded on $X$
is a constant. We say that a disjoint sequence of intervals
$\{I_n\}$ on the real line is \itshape a long sequence of intervals
\normalfont if
$$
\sum_n\frac{|I_n|^2}{1+\dist^2(0,I_n)}=+\infty.
$$

A complete solution of the P\'olya problem was obtained by
Mishko Mitkovski and Alexei Poltoratski \cite{mp}. In particular a separated sequence
$X\subset\mathbb{R}$ is not a P\'olya sequence if and only
if there exists a long sequence of intervals $\{I_n\}$ such that
$$
\frac{\card(X \cap I_n)}{|I_n|}\rightarrow 0.
$$

Applying this result to our $R$ and $\Sigma$ (formally speaking, $\Sigma$ is not a separated sequence but by construction it is a union of two separated
sequences which are interlacing), we find a long system
of intervals $\{I_n\}$ such that
$$
\frac{\card(\Sigma\cap I_n)}{|I_n|}\rightarrow0.
$$

Given $I=[a,b]$, denote $I^-=[a,(2a+b)/3]$, $I^+=[(a+2b)/3,b]$,
$$
\Delta^*_I = \inf_{I^+}[\pi \varepsilon x -\pi n_\Sigma(x) +v] -
\sup_{I^-}[\pi\varepsilon x - \pi n_\Sigma(x) +v].
$$
Now, for a long system of intervals $\{I_n\}$ and for some $c>0$
we have
$$
\Delta^*_{I_n}\geq c|I_n|.
$$

Next we use a version of the second Beurling--Malliavin theorem
given by N. Makarov and A. Poltoratski in \cite{pm}.

Combining Proposition 3.13 and Theorem 5.9 from \cite{pm} we get

\begin{proposition} Suppose $\gamma\in C(\mathbb{R})$.
If there exists $c>0$ and a long system of intervals $I_n$ such that
\begin{equation}
\Delta^*_{I_n}[\gamma]\geq c|I_n|,
\label{delta}
\end{equation}
then $\gamma$ cannot be represented as $\alpha+\widetilde{h}$,
where $\alpha$ is decreasing and $h\in L^1((1+x^2)^{-1}dx)$.
\label{deltaprop}
\end{proposition}
If we apply this for the function $\pi\varepsilon x - \pi n_\Sigma(x) +v$
we arrive to a contradiction. Proposition \ref{mainprop}
is completely proved.
\qedsymbol

\begin{remark}
We close this subsection with an additional explanation of  the
result in Proposition \ref{deltaprop}. Assume that
$|I_n|=o(\dist(0, I_n))$. Let $h_n=h\chi_{10I_n}$ be a restriction
of function $h$ onto interval $10I_n$ (the interval of the length
$10|I_n|$ and with the same center as $I_n$). If the inequality
\eqref{delta} holds for $h_n$, then the Kolmogorov theorem states
that there exists $c$ such that
$$\int_{|\widetilde{h_n}|>A}\frac{dx}{1+x^2}\leq \frac{c}{A}\int_{\mathbb{R}}\frac{|h_n(x)|dx}{1+x^2},\quad A>1.$$
Choosing $A=\varepsilon|I_n|$, $\varepsilon>0$ we have
$$\frac{|I_n|^2}{1+\dist^2(0,I_n)}\leq \frac{3c}{\varepsilon} \int_{\mathbb{R}}\frac{|h_n(x)|dx}{1+x^2}.$$
Summing up over $n$ we get the contradiction
$$\infty=\sum_n\frac{|I_n|^2}{1+\dist^2(0,I_n)}\leq C_1\sum_n\int_{\mathbb{R}}\frac{|h_n(x)|dx}{1+x^2}\leq C_2\int_{\mathbb{R}}\frac{|h(x)|dx}{1+x^2}<\infty.$$
We refer to \cite{pm} for the details.
\end{remark}

\subsection{A reformulation in terms of an approximation result.} Given a distribution
$\varphi$ and  $w\in \mathbb{C}$ which is a zero of order $n$ of
its Fourier transform $\hat{\varphi}$, we denote by
$\varphi_{w,k}$, $1\le k\le n$, the distributions with Fourier
transforms
$$\hat{\varphi}_{w,k}(z)=\frac{\hat{\varphi}(z)}{(z-w)^k}.$$ As a
consequence of the proof of Theorem \ref{main1} we have the
following result.
\begin{corollary}\label{approxres}
Let $\varphi$ be a compactly supported distribution, let $I$ be
the convex hull of its support, and let $\Lambda$ be a subset of
the zero set of $\hat{\varphi}$ (counting multiplicities).
If$$D^+(\Lambda)<\frac{|I|}{2\pi},$$ then every distribution
$\psi$ with support in $I$ whose Fourier transform vanishes on
$\Lambda$, lies in the weak-star closure of the linear span of
$\{\varphi_{w,k}:~w\notin \Lambda\}$.
\end{corollary}


\section{Proof of Theorem \ref{main2}\label{examp}}

Let $\PP$ be the set of all polynomials.
\begin{lemma}
\label{debr}
There exists a sequence $\Lambda \in \RR$ of density $\pi$
with the generating function $G_\Lambda$ and an entire function $S$ such that
the following three conditions hold\textup:

$(i)$  $G_\Lambda \PP \subset \pw$ and $G_\Lambda S \in \pw$\textup;

$(ii)$ $G_\Lambda S = \sum_{n\in \Z} a_n k_n$ and for any $N>0$
we have $a_n = o(n^{-N})$, $|n|\to\infty$\textup;

$(iii)$ $G_\Lambda S$ is orthogonal to $G_\Lambda \PP$ in $\pw$.
\end{lemma}

Assume for the moment that Lemma \ref{debr} is proved. We first show how to deduce the
counterexample of Theorem \ref{main2} from this lemma.
\bigskip
\\
{\it Proof of Theorem \ref{main2}.}
Assume that $\Lambda$ and $S$ are constructed. Put
$$
M = \big\{f\in L^2(-\pi, \pi):\ \hat f \in G_\Lambda \PP  \big\}
$$
(recall that any function $F\in \pw$ is of the form $F = \hat f$
for some $f\in L^2(-\pi, \pi)$).
Of course, each element $f\in M$ defines a continuous linear functional on
$C^\infty(-2\pi, 2\pi)$ by
$$
\phi_f(h) = \int_{-\pi}^\pi h(t) \overline{f(t)} dt, \qquad h \in C^\infty(-2\pi, 2\pi).
$$
The functional $\phi_f$ is well defined since $f(t)\equiv 0$, $|t|\ge \pi$.

Now let
$$
L = M^\perp = \{h\in C^\infty(-2\pi, 2\pi):\ \phi_f(h) = 0, \, f\in M\}.
$$
By the construction, $L$ is a closed subspace of
$C^\infty(-2\pi, 2\pi)$ and
$$
\{f\in C^\infty(-2\pi,2\pi): f|_{[-\pi, \pi]} \equiv 0\} \subset L.
$$
Also, since the set of common zeros of $\widehat{L^\perp}$ coincides with $\Lambda$
we have $\sigma(D|_L) = \Lambda$.

Let us show that $L$ is $D$-invariant. Let $h\in L$.
We need to show that $\int_{-\pi}^\pi h'(t)\overline{f(t)}dt =0$ for any $f\in M$.
Since $f$ vanishes outside $[-\pi, \pi]$, the integral depends only on the values of
$h$ inside this interval. Thus we may assume without loss of generality that
$\supp h \subset (-\pi-\vep, \pi+\vep)$ for some small $\vep>0$.
Therefore both $F = \hat f$ and $H = \hat h$ are rapidly decaying functions and we have
$$
\int_{-\pi}^\pi h'(t) \overline{f(t)} dt = \int_\RR x H(x)\overline{F(x)}dx.
$$
We have $F = G_\Lambda P$ for some polynomial $P$. Then $xF(x) =  xP(x)G_\Lambda(x)
= \hat f_1$ for some $f_1\in M$. Hence,
$$
\int_\RR x H(x) \overline{F(x)} dx = \int_{-\pi}^\pi h(t) \overline{f_1(t)}dt = 0,
$$
since $h\in M^\perp$. We have seen that $L$ is $D$-invariant.

Now we construct a continuous functional $\phi$ on $L$ such that
$\phi|_{L_0} = 0$, but $\phi|_L \ne 0$, where, as in the proof of the first theorem,
$$L_0= \overline{L_{res}+\mathcal{E}(L)}.$$
Let $h_0\in L^2(-\pi, \pi)$ be such that $\widehat{\overline{h_0}} = G_\Lambda S$.
Recall that $G_\Lambda S = \sum_{n\in \Z} a_n k_n$ where $a_n = o(n^{-N})$ for any $N>0$.
Hence, $\overline{h_0(t)} = \sum_{n\in \Z} a_n e^{int}$ and, by
the fast decay of $a_n$ we conclude
that $h_0$ is a $C^\infty$ function in the {\it closed} interval $[-\pi, \pi]$.
Denote by $h$ some function in $C^\infty(-2\pi, 2\pi)$ such that
$h|_{[-\pi, \pi]} = h_0$.

Consider the functional $\phi$ on $L$,
$\phi(g) = \int_{-\pi}^\pi g(t)\overline {h_0(t)} dt$. It is clear that $\phi$
annihilates the set $\{g\in C^\infty(-2\pi,2\pi): g|_{[-\pi, \pi]} \equiv 0\}$.
Also, $\phi(e^{i\lambda t}) = \widehat{\overline{h_0}}(\lambda) = 0$.
Thus, $\phi$ annihilates $L_0$.

Let us show that $h \in L$. Since  $\phi(h) = \int_{-\pi}^\pi |h_0(t)|^2 dt >0$,
we then conclude that $L\ne L_0$. Indeed, for any $f\in M$
we have
$$
\begin{aligned}
\phi_f(h) & = \int_{-\pi}^\pi h(t)\overline{f(t)} dt
= \int_{-\pi}^\pi h_0(t)\overline{f(t)} dt \\
& = \int_\RR \hat h_0(x)\overline{\hat{f}(x)} dx = 0,
\end{aligned}
$$
since $\hat h_0 = G_\Lambda S$, $\hat f = G_\Lambda P$ for some polynomial $P$,
and $G_\Lambda S \perp G_\Lambda \PP$ in $L^2(\RR)$
by the construction of Lemma \ref{debr}.
Thus we have constructed a functional which separates $L_0$ and $L$.
\qed
\bigskip
\\
{\it Proof of Lemma \ref{debr}.}
Recall that we denote by $\ZL_F$ the zero set of an entire
function $F$ (all functions involved will have simple zeros). Let
$$
U(z) = \frac{\sin (\pi \sqrt{z})}{\pi \sqrt{z}} = \prod_{n\in \N}
\bigg(1-\frac{z}{n^2}\bigg),
$$
and let $V$ be some product with very lacunary zeros, say
$$
V(z) = \prod_{n\in \N} \bigg(1-\frac{z}{2^{2n}+1}\bigg).
$$
Put
$$
G_\Lambda(z) = \frac{\sin \pi z}{U(z)V(z)}.
$$
Note that $UV$ tends to infinity faster than any power along $\RR$
except some small neighborhoods of the zeros of $UV$. Therefore, using, e.g.,
Plancherel--P\'olya theorem one can easily show that $G_\Lambda P \in \pw$
for any polynomial $P$.

Let us introduce two more entire functions
$$
T(z) = \frac{\sin (\pi \sqrt{z+1/2}/10)}{W(z) \sqrt{z+1/2}}, \qquad
W(z) = \prod_{n\in \N} \bigg(1-\frac{z}{100\cdot 2^{2n}-1/2}\bigg)
$$
(note that the zeros of the nominator are exactly of the form $100n^2-1/2$, $n\in \N$).

Now we define the sequence $a_n$ by $a_n :=0$ for $n\notin \ZL_U\cup\ZL_V$ and
$$
a_n:= (-1)^n T(n), \qquad n \in \ZL_U\cup\ZL_V.
$$
It is easy to see that $T(n) = o(n^{-N})$ for any $N$
when $n\to +\infty$, since $\sin(\pi \sqrt{x})$ is bounded for $x>0$ and $W(n)$
grows faster than any power along $\N$.

Now we define the function $S$ by the formula
\begin{equation}
\label{nb}
\frac{G_\Lambda(z)S(z)}{\sin \pi z} = \frac{1}{\pi}
\sum_{n\in \mathbb{Z}}\frac{(-1)^n a_n}{z-n}.
\end{equation}
Thus $G_\Lambda S = \sum_{n\in \Z} a_n k_n$, where $k_n$ are the elements of the
orthogonal basis of reproducing kernels of $\pw$ (the Shannon--Kotelnikov formula),
and so $G_\Lambda S \in \pw$. This proves $(i)$ and $(ii)$, since
$\{a_n\}$ has a fast decay.

Note that the summation goes only along $n \in \ZL_U\cup\ZL_V$, and we can rewrite
\eqref{nb} as the following interpolation formula:
$$
\frac{S(z)}{U(z)V(z)} = \frac{1}{\pi} \sum_{n \in \ZL_U\cup\ZL_V} \frac{T(n)}{z-n}.
$$
Next we put $S_1 = G_\Lambda T$. We need to show that the following
interpolation formula also holds:
$$
\frac{S_1(z)}{\sin \pi z} = \frac{1}{\pi}\sum_{n\in \Z} \frac{a_n G_\Lambda(n)}{z-n}.
$$
This formula can be rewritten in the following way using the fact that
$a_n = (-1)^n T(n)$, and $G_\Lambda(n) = \pi (-1)^n \big((UV)'(n)\big)^{-1}$,
$n \in \ZL_U\cup\ZL_V$:
\begin{equation}
\label{nb1}
\frac{T(z)}{U(z)V(z)} = \sum_{n\in \ZL_U\cup\ZL_V} \frac{T(n)}{(UV)'(n)(z-n)}.
\end{equation}

We have already mentioned that $T(n)$ decays faster than any power
when $n\to\infty$. It is also easy to see that
$|(UV)'(n)|\to \infty$ when $n\to\infty$ and $n\in \ZL_U\cup\ZL_V$.
Since $V$ is a lacunary product it
is clear that $|V(n)|$ is large for $n \in \ZL_U$ and $|V'(n)|$ is large for
$n\in \ZL_V$, while $|U(n)|$, $n\in \ZL_V$ and $|U'(n)|$, $n\in \ZL_U$,
have a power-type below estimate in $n$. Thus, the series in the right-hand side
of \eqref{nb1} converges uniformly on compact sets separated from the poles.

Clearly the residues coincide, and so the difference between the left
and the right-hand side of \eqref{nb1} (let us denote it by $H$)
is an entire function. Obviously. $H$ is of zero exponential type. It remains
to notice that the function $T/(UV)$
in the left-hand side tends to zero along the imaginary axis
(for the function on the right this is obvious), thus $|H(iy)|\to 0$, $|y|\to \infty$,
whence $H\equiv 0$.

By exactly the same arguments we may show that for any $P\in \PP$,
$$
\frac{P(z) S_1(z)}{\sin \pi z} = \frac{1}{\pi}\sum_{n\in \Z}
\frac{P(n) a_n G_\Lambda(n)}{z-n}.
$$
It remains to verify that the function $G_\Lambda S = \sum_{n\in \Z} a_n k_n$
is orthogonal to all functions of the form $z^k G_\Lambda $, $k\in \Z_+$.
Let $a\notin \Z$ and let $P(z) =(z-a)z^k$. Since $k_n$
are the reproducing kernels of $\pw$, we have
$$
(z^k G_\Lambda, G_\Lambda S) =
\frac{1}{\pi}\sum_{n\in \Z} n^k a_n G_\Lambda(n) =
\frac{1}{\pi} \sum_{n\in \Z} \frac{P(n) a_n G_\Lambda(n)}{n-a} =
\frac{P(z)S_1(z)}{\sin \pi z}\bigg|_{z=a} = 0.
$$
This proves $(iii)$ and completes the proof of the lemma.
\qed


{ \section{Subspaces with non-compact residual interval \label{nores}}

In this section we prove Theorem \ref{main3}.

We begin with a simple observation related to the absence
of the residual part. Given a distribution $\varphi$ compactly
supported on $(a,b)$ we denote throughout this section by
$I(\varphi)$ the convex hull of its support
and by $|I(\varphi)|$ its length. Note that if $L$ has a residual interval $I$ strictly contained in $(a,b)$ then all distributions in the annihilator of $L$ are supported on this interval, and clearly,
$$
\sup \{|I(\varphi)|:~\varphi\in L^\perp\}=|I|.
$$
the next lemma shows that this observation remains true when $I=(a,b)$ as well.}
\begin{lemma}
If $\sigma(D|_L)$ is an infinite discrete subset of $\mathbb{C}$
and $L_{res}=\{0\}$, then
$$
\sup \{|I(\varphi)|:~\varphi\in L^\perp\}=b-a,
$$
where $b-a=\infty$ if $(a,b)$ has infinite length.
\label{lemma1}
\end{lemma}

\begin{proof}
Let $s$ be the supremum in the statement. Under the assumption on $\sigma(D|_L)$
it follows easily that $s>0$.
Also note that if $|I(\varphi)|,|I(\psi)|>s/2$,
then these intervals must have a nonempty intersection
otherwise $|I(\varphi+\psi)|>s$. Thus,
$$
I=\bigcup\{I(\varphi):~\varphi\in L^\perp,~|I(\varphi)|>s/2\}
$$
is an interval with $|I|\ge s$. It is easy to see that
we must have $|I|= s$. Indeed, if $c,d\in I$ with $d-c>s$ such that
$c\in I(\varphi),~d\in I(\psi)$, then clearly, $|I(\varphi+\psi)|>s$,
which is a contradiction.

We claim that
$I$ contains the interior of any interval $I(\varphi)$, with $\varphi\in L^\perp$.
Assume the contrary, i.e., that
there exists $\varphi\in L^\perp$ such that the
interior of $I(\varphi)$ is not contained in $I$.
If $J$ is a nontrivial interval in $I(\varphi)\setminus I$,
we choose $\psi\in L^\perp$ with
$$
|I(\psi)|> \max\{s-|J|, s/2\}
$$
and note again that in this case $|I(\varphi+\psi)|>s$, which is a contradiction.

The claim implies that all distributions $\varphi\in L^\perp$
are supported on the closure of $I$. If $s=|I|<b-a$,
then $L$ contains $L_I$ which contradicts our assumption.
It remains that $s=b-a$, and the lemma is proved.
\end{proof}


\subsection{Radius of completeness} Let $\Lambda$ be a discrete subset
of $\mathbb{C}$. Put
$$
R(\Lambda)=\sup\{a: \mathcal{E}(\Lambda) \text{ is complete in } L^2[0,a]\}.
$$
The number $R(\Lambda)$ is called the {\itshape radius of
completeness of} $\Lambda$. It is well known that $R(\Lambda)$ is
equal to the Beurling--Malliavin (effective) density of $\Lambda$
\cite{BM2}, but we will not use this remarkable fact. The
following observation is important for our purposes.
\begin{remark}
The conclusions of Theorem \ref{main1},  Proposition \ref{mainprop}, and Corollary \ref{approxres} continue to hold if  we replace in the hypothesis the upper density by the radius of completeness.
\label{rem1}
\end{remark}

\begin{proof}
The condition $D_+(\Lambda)<\frac{|I|}{2\pi}$ is used only in the
proof of Proposition \ref{mainprop}, in order to show that
equation \eqref{n15} implies equation \eqref{n16}. We claim that
this implication remains true under the assumption that
$R(\Lambda)<|I|$. To see this, suppose again that $I=[-\pi,\pi]$.
If $R(\Lambda)<2\pi$, then for any $\varepsilon>0$ there exists a
nontrivial entire function $T$ such that
$G_2T\in\mathcal{PW}_{\pi-\varepsilon}$, where $G_2$ is a
canonical product corresponding to the sequence $\Lambda$. Without
loss of generality we can assume that the zeros of $T$ are in
$\mathbb{C}^+$. Then
$$
\arg G_2^\#+\arg T=\pi(1-\varepsilon)x+\tilde{u}+c,\quad u\in L^1((1+x^2)^{-1}),\quad c\in\mathbb{R},
$$
and this implies \eqref{n16}.
\end{proof}


\subsection{Proof of Theorem \ref{main3}.} { Let $I$ be the residual interval of $L$. 
If $R(\Lambda)\geq |I|$, then $\mathcal{E}(\Lambda)$ is complete in  $L^2(J)$ for any compact subinterval of $J$ of $I$, and consequently,
 $L$ can not be annihilated by any
distribution with  compact support contained in $I$, i.e. $L=C^\infty(a,b)$.}

 If $R(\Lambda)<|I|$, we can use the considerations at the beginning of this section together with  Lemma \ref{lemma1} to conclude
that given $\varepsilon >0$, there exists $\varphi\in L^\perp$
with $|I(\varphi)|>|I|-\varepsilon$. Then by the modified version
of  Corollary \ref{approxres} stated in the remark above,
$L^\perp$ contains all distributions with  support in
$I(\varphi)$, whose Fourier transform vanishes at the points of
$\sigma(D|_L)$ with the appropriate multiplicities. Since
$\varepsilon$ is arbitrary, the result follows. $\square$

\end{document}